\documentclass[12pt]{article}
\usepackage{amsmath}
\usepackage{graphicx,psfrag,epsf}
\usepackage{enumerate}
\usepackage{natbib}
\usepackage{amsmath,amsthm,amsfonts,epsfig,amssymb,natbib,eucal,eufrak}

\usepackage{booktabs}
\usepackage{multirow}
\usepackage{siunitx}
\usepackage{qtree}
\usepackage{hyperref}

\usepackage{float}
\restylefloat{table}
\usepackage{color}

\usepackage{pdfpages}
\usepackage{wrapfig}

\usepackage{comment}

\usepackage{tikz}

\newtheorem*{thm}{Theorem}

\newtheorem{remark}{Remark}

\newcommand{\blind}{0}

\begin{document}

\def\spacingset#1{\renewcommand{\baselinestretch}%
{#1}\small\normalsize} \spacingset{1}


\if0\blind
{
  \title{\bf On Round-Robin Tournaments with a Unique Maximum Score}

  \author{
  Yaakov Malinovsky
    \thanks{email: yaakovm@umbc.edu}
   \\
    Department of Mathematics and Statistics\\ University of Maryland, Baltimore County, Baltimore, MD
    21250, USA\\
    \\
  John W. Moon
    \thanks{email: jwmoon@ualberta.ca}
   \\
    Department of Mathematical and Statistical Sciences\\ University of Alberta, Edmonton, AB T6G 2G1,
    Canada\\
}
  \maketitle
} \fi

\if1\blind
{
  \bigskip
  \bigskip
  \bigskip
  \begin{center}
    {\LARGE\bf Title}
\end{center}
  \medskip
} \fi

\begin{abstract}
Richard Arnold Epstein (1927-2016) published the first edition of "The Theory of Gambling and
Statistical Logic" in
1967.
He introduced some material on round-robin tournaments (complete oriented graphs) with $n$ labeled
vertices in
Chapter 9; in particular, he stated, without proof, that the probability that there is a unique
vertex with the
maximum score
tends to $1$ as $n$ tends to infinity.
Our goal here is to give a proof of this result along with
some historical remarks and comments.

\end{abstract}

\noindent%
{\it Keywords: Complete graph, large deviations, maximal score, probabilistic inequalities, random graph, round-robin tournaments

\noindent%
}
\spacingset{1.45} 

\newpage
\section{Introduction}
\label{se:Int}
In a classical round-robin tournament, each of $n$ players wins or loses a game against each of the
other $n-1$
players \citep{Moon1968}. Let $X_{ij}$ equal $1$ or $0$ according to whether player $i$ wins or loses the game
played against
player $j$, for
$1 \leq i, j \leq n$, $i \neq j$, where $X_{ij} + X_{ji} = 1$. We assume that all ${n \choose 2}$
pairs $(X_{ij},
X_{ji})$ are independently distributed with $P(X_{ij} = 1) =P(X_{ji} = 0) = 1/2$. Let
\begin{equation*}
s_{i}:=s_i(n)=\sum_{j=1, j\neq i}^{n}{X_{ij}}
\end{equation*}
denote the score of player $i, 1\leq i \leq n$, after playing against all the other $n-1$ players. We
refer to $\left(s_1, s_2,\ldots, s_n\right)$ as the score sequence of the tournament.
When the identity of the players with a given score in a tournament is of no particular significance, we often rearrange the elements of the score sequence in nondecreasing order and refer to this rearranged sequence as the nondecreasing score sequence of the tournament.
These tournaments can be represented by complete oriented graphs in which the vertices represent the players and each pair of distinct vertices $i$ and $j$ is joined by an edge oriented from $i$ to $j$ or from $j$ to $i$ according to whether $X_{ij} = 1$ or $X_{ji} = 1$. The frequency of a given $n$-vertex nondecreasing score sequence is the sum of the frequencies (or the total number) of all labelled n-vertex tournaments that contain the same number of vertices of score $k$ as the number of elements of value $k$ in the given nondecreasing sequence, for $0\leq k\leq n-1$.

Round-robin tournaments can be considered as a model of paired comparison experiments used in an
attempt to rank a number of objects with respect to some criterion--or at least to determine if there is any significant difference between the objects--when it is impracticable to compare all the objects simultaneously: see, e.g., \cite{Z1929}, \cite{D1988}, \cite{DE2001}, and \cite{AK2022}. In particular, \cite{D1959} generated the
score sequences
of tournaments with $n$ players for $3 \leq n \leq 8$ and their frequencies by expanding products of
the form
\begin{equation*}
F(n)=\prod_{1\leq i <j\leq n}\left(\frac{1}{2}w_i+\frac{1}{2}w_j\right).
\end{equation*}

For example,
\begin{equation*}
F(3)=\frac{1}{2^3}\left(w_1^2w_2+w_1w_2^2+w_1^2w_3+w_1w_3^2+w_2^2w_3+w_2w_3^2+2w_1w_2w_3\right);
\end{equation*}
so the joint probability mass function of the scores are
\begin{align*}
&
P\left(s_1=2, s_2=1, s_3=0\right)=\cdots=P\left(s_1=0, s_2=1, s_3=2\right)=\frac{1}{8},\,\,\,\, \text{and}\\
&
P\left(s_1=1, s_2=1, s_3=1\right)=\frac{2}{8};
\end{align*}
and the frequencies of the nondecreasing score sequences $(0, 1, 2)$ and $(1, 1, 1)$ are six and two, respectively.
\cite{D1959} used this information to develop, among other things, tests for deciding
whether the maximum
in a given outcome was significantly larger than the expected value $(n-1)/2$ of a given score.

Let $r_n$ denote the probability that an ordinary tournament with $n$ labeled vertices has a unique
vertex with maximum score. \citet[p.~353]{E1967}
gave the values $r_4 = .5, r_5 = .586, r_6 = .627, r_7 = .581$, and $r_8 = .634$ without further explanation.
However, a reference to \cite{D1959} is given a few pages later, so presumably he deduced these values for $r_i$ from Table 1 in \cite{D1959}, except for one error: the value for $r_8$ should have been
$160,241,152/2^{28}=.5969\cdots$.
\cite{S2022} has recently
pointed  out that
\cite{M1923} generated the score sequences and their frequencies for tournaments with up to $9$
vertices and his results agree with David's for $n=8$. It follows
from MacMahon's data that $r_9=42,129,744,768/2^{36}=.6130\cdots$.

Epstein also stated, without a proof or reference, that as $n$ increases indefinitely, $r_n$ approaches unity.
Some later editions of his book contain more material on tournaments but the material on $r_n$ remains unchanged. A survey paper by \cite{G1984}, on various unsolved problems, mentions Epstein's problem on $r_n$ as being still unsolved.

In the next section we shall give some additional numerical data and some simulation results that illustrate the behavior of $r_n$. In Section \ref{sec:3} we show that Epstein's Conjecture is indeed correct.

\section{Numerical Data and Monte-Carlo Simulations}
\label{sec:2}

As a partial check, we obtained the same value for $r_9$, as given earlier, by determining the number of ways of constructing 9-vertex tournaments with a unique vertex $v$ of maximum score by adjoining $v$ to an 8-vertex tournament with any given 8-vertex nondecreasing score sequence. We then tried to determine the value of $r_{10}$, which was unknown to us at the time, in the same way from information about the 9-vertex case. In doing this we discovered that MacMahon's values of $361,297,520$ for the nondecreasing score sequence $(2, 2, 3, 3, 4, 4, 6, 6, 6)$ and its complement were incorrect; these frequencies should have been divisible by $9$, since there are $9$ choices for the label of the winner of the match between the two vertices of score $4$ (in both cases). These two frequency values should each be increased by $10,000$; and, using these corrected values we found that $r_{10}=21,293,228,876,800/2^{45} = .6051\cdots$.

During this process we discovered that Doron Zeilberger \citep{Z2016} had extended MacMahon's work and had generated the nondecreasing score sequences and their frequencies for tournaments with up to 15 vertices using the Maple program. (We remark that Zeilberger's frequencies for the two sequences mentioned earlier agree with the corrections we gave.)
Using a Matlab program, the values of $r_n$ were deduced from Zeilberger's data for $n=10$, $11$, and $12$.
The value obtained when $n=10$ agreed with the value stated above.
And, as a partial check, we confirmed that the value for $r_{12}$ obtained by using $12$-vertex frequency data can also be deduced from the $11$-vertex data.
The values for $n=4,5,\ldots,12$ are given in Table 1. We also learned that the sequence \url{https://oeis.org/A013976} contains the values of the number of tournaments with a unique vertex of maximal score for $n = 1, 2,\ldots,16$; these values are attributed to Michael Stob and Andrew Howroyd.

It is not feasible to test Epstein's claim  for large values of $n$ by generating score sequences and their frequencies directly because of the length of time this would take.
For example, executing the \cite{Z2016} Maple code for $n=17$ on a powerful computer (a dual CPU Intel 2620 v4, 1TB of memory, Unix operating systems) took 157:04 hours. We note that in this case, the number of different nondecreasing score sequences is 6,157,068 (see \url{https://oeis.org/A000571} and references therein.)
So we have used Monte-Carlo simulations \citep{MU1949} to test Epstein's statement for larger values of $n$. For a
given value of $n$ we sample $n(n-1)/2$ values of random Bernoulli variables $X_{ij}, 1 \leq i < j \leq n$, where
$P(X_{ij}=0)=P(X_{ij}=1)=1/2$; this determines a random n-vertex  tournament and its score sequence. We repeated this process $M$ times for a predetermined integer $M$. We let $I_t$ denote a random indicator function that equals one
if the tournament obtained at the t-th repetition has a unique score of maximum value, and equals zero otherwise, for $1 \leq t \leq M$. Then $\hat{r}_{n}(M)=1/M\sum_{t=1}^{M}I_t$ is an unbiased estimator of $r_n$
(i.e., $E\left(\hat{r}_{n}(M)\right)=r_n$ for any $M$); also $\hat{r}_n$ is a consistent estimator of $r_n$ for large $M$, (i.e., $\lim_{M\rightarrow \infty}P\left(|\hat{r}_{n}(M)-r_n|>\varepsilon\right)=0$
for any $\varepsilon>0$).
A consistent estimator ${\displaystyle \hat{\sigma}\left(\hat{r}_{n}(M)\right)}$ of the standard deviation of ${\displaystyle \hat{r}_{n}(M)}$  is ${\displaystyle \hat{\sigma}\left(\hat{r}_{n}(M)\right)=\left({\hat{r}_{n}(M)(1-\hat{r}_{n}(M))}/{M}\right)^{0.5}}$.

We used smaller values of $M$ for some of the larger values of $n$ because of time constraints. The
results of these
simulations are given in the Table 1 below.

\begin{table}[H]
\label{tab:1}
\caption{$r_4,\ldots,r_8$ were calculated from the scores distribution given in Table 1 of
\cite{D1959}; $r_9$ from \cite{M1923} data; $r_{10}, r_{11}, r_{12}$ from \cite{Z2016} data; see also \url{https://oeis.org/A013976}.}
\begin{center}
\begin{tabular}{lllllllll}
  $n$ & $M$ &$r_n$ & $\hat{r}_n(M)$ & $10^3\hat{\sigma}\left(\hat{r}_{n}(M)\right)$ \\
  \hline
  4 &       $10^6$          &  $0.5$                                  & $0.5003\cdots$            & $0.499\cdots$      &\\

  5 &       $10^6$          &  $600/2^{10}=0.5859\cdots$              & $0.5862\cdots$            & $0.492\cdots$      &\\

  6 &       $10^6$          &  $20,544/2^{15}= 0.6269\cdots$          & $0.6267\cdots$            & $0.483\cdots$      &\\

  7 &       $10^6$          &  $1,218,224/2^{21}=0.5808\cdots$        & $0.5815\cdots$            & $0.493\cdots$      &\\

  8 &       $10^6$          &  $160,241,152/2^{28}=0.5969\cdots$      & $0.5965\cdots$            & $0.490\cdots$      &\\

  9 &       $10^6$          & $42,129,744,768/2^{36}=0.6130\cdots$    & $0.6132\cdots$            & $0.487\cdots$      &\\

  10&       $10^6$          & $21,293,228,876,800/2^{45}=0.6051\cdots$    & $0.6053\cdots$        & $0.488\cdots$      &\\

  11&       $10^6$          & $22,220,602,090,444,032/2^{55}=0.6167\cdots$& $0.6164\cdots$        & $0.486\cdots$      &\\

  12&       $10^6$          & $45,959,959,305,969,143,808/2^{66}=0.6228\cdots$    & $0.6231\cdots$  & $0.484\cdots$ &\\

  13&       $10^6$          &                                      & $0.6236\cdots$      & $0.484\cdots$ &\\
  14&       $10^6$          &                                      & $0.6325\cdots$      & $0.482\cdots$ &\\

  15&       $10^6$    &                                           & $0.6364\cdots$      & $0.481\cdots$ &\\


  30&       $10^6$          &                                      & $0.6903\cdots$      & $0.462\cdots$ &\\

  50&       $10^6$          &                                      & $0.7299\cdots$      & $0.444\cdots$      &\\

  100&      $10^6$          &                                      & $0.7797\cdots$      & $0.414\cdots$       &\\

  500&      $10^6$          &                                      & $0.8746\cdots$       & $0.331\cdots$      &\\

  1,000&    $10^6$          &                                      & $0.9032\cdots$       & $0.295\cdots$      &\\

  10,000&   $10^5$           &                                      & $0.9623\cdots$       & $0.601\cdots$      &\\
  100,000&  $10^3$           &                                      & $0.986$       & $3.715\cdots$      &\\

\end{tabular}
\end{center}
\end{table}


\section{The Uniqueness of the Maximum Score}
\label{sec:3}
\subsection{Useful Facts and Notation}
Let  $(p_{ij})$ denote a probability matrix such that $p_{ij} + p_{ji} = 1$, and $p_{ij}=P(X_{ij} = 1)$ for $1 \le i \le j \le n$, $p_{ii}=0$ for $1 \le i \le n$; and where the variables $X_{ij}$ and $X_{ji}$ are as defined in Section \ref{se:Int}.
\cite{H1963} used a coupling argument to establish the following inequality for the joint distribution function of the scores $s_1,\ldots,s_n$ in  a round-robin tournament:

\begin{equation}
\label{eq:H}
P\left(s_1<k_1,\ldots, s_m<k_m\right)\leq P\left(s_1< k_1\right)\cdots P\left(s_m<k_m\right),
\end{equation}
where $m \le n$, for any such probability matrix $(p_{ij})$, and any numbers $\left(k_1,\ldots,k_m\right)$;
the inequality also holds if the $<$ sign is replaced by the $\leq$ sign throughout.  We assume here that $p_{ij}=1/2$ for all $i\leq j$. As we shall see presently, Huber's inequality has implications for the maximum scores in tournaments. We wrote Professor Noga Allon about Epstein's conjecture and he \citep{A2022} referred us to a paper by \cite{EW1977} that considered the analogous problem for the vertices of maximum degree in a random labelled graph in which pairs of distinct vertices are joined by an edge with probability $1/2$.

For expository convenience we introduce some notation and relations that we shall need later. Let
\begin{align*}
&
b(n-1, j)=P(s_i = j)={{n-1}\choose{j}}\frac{1}{2^{n-1}}
\end{align*}
and
\begin{align*}
&
B(n-1, j)=P(s_i > j) = \sum_{k > j}b(n-1, k)
\end{align*}
for $0 \leq j \leq n-1$ and $1 \leq i \leq n$.

Next, let
\begin{equation}
\label{eq:tn}
t_{n-1}=\lceil(n-1)/2 + x_{n-1}((n-1)/4)^{1/2}\rceil
\end{equation}
where
\begin{equation}
\label{eq:xn}
x_{n-1}=(2\log(n-1)-(1+\epsilon)\log(\log(n-1)))^{1/2}
\end{equation}
for an arbitrary constant $\epsilon$ between $0$ and $1$, say. It is not difficult to see that
\begin{align}
\label{eq:NEW}
x_{n - 1}\leq\left((t_{n - 1} - (n - 1)/2\right)((n - 1)/4)^{-1/2}\leq  x_{n - 1} + ((n - 1)/4)^{-1/2}.
\end{align}
It follows from \eqref{eq:NEW} and definition \eqref{eq:xn} that $x_{n-1}\rightarrow \infty$ and  $x_{n-1}=o(n^{1/6})$ as $n\rightarrow \infty$, and the same conclusion holds when $x_{n-1}$ is replaced by
\begin{align*}
\left((t_{n-1} - (n - 1)/2\right)\left((n - 1)/4\right)^{-1/2}.
\end{align*}

Consequently, we may appeal to relation (4.5.1) in \citet[p.~204]{R1970} and relations (2.7) and (6.7) in \citet[pp.~180 \& 193]{F1968} to conclude that

\begin{equation}
\label{eq:bn}
b(n-1, t_{n-1})\sim \left(\frac{2}{\pi(n-1)}\right)^{1/2}e^{-\frac{x_{n-1}^2}{2}}
\sim \frac{\sqrt{2}(\log(n-1))^{(1+\epsilon)/2}}{\sqrt{\pi(n-1)^3}}
\end{equation}
and

\begin{equation}
\label{eq:Bn}
B(n-1, t_{n-1})\sim \frac{1}{\sqrt{2\pi}}
\frac{1}{x_{n-1}} e^{-\frac{x_{n-1}^2}{2}}\sim \frac{(\log(n-1))^{\epsilon/2}}{\sqrt{4\pi}(n-1)}.
\end{equation}

\subsection{Main Result}
\begin{thm}
The probability that a random n-vertex tournament $T_n$ has a unique vertex of maximum score tends to
$1$ as $n$ tends
to infinity. In particular, if $t_{n-1}$ is defined as in \eqref{eq:tn} and \eqref{eq:xn} and
$s^{\star}$ denotes the
maximum value of the scores $s_1,\ldots, s_n$ in $T_n$, then the following statements hold:

\begin{enumerate}
\item[(i)] [\cite{H1963}]
$P(s^{\star}>t_{n-1})\rightarrow 1$ as $n \rightarrow \infty$.
\item[(ii)] If $W_n=W_n(T_n)$ denotes the number of  ordered pairs of distinct vertices $u$ and $v$
    in $T_n$ such
    that
$s_u=s_v=h$ for some integer $h$ such that   $t_{n-1}  <  h   \leq  n-1$, then
$P(W_n > 0)\rightarrow 0$ as $n \rightarrow \infty$.
\end{enumerate}
\end{thm}

\begin{proof}\,\\
{\bf First Proof of $(i)$}\\
\cite{H1963} observed that the required conclusion follows from the facts that
\begin{align}
\label{eq:Hln}
&
P\left(s^{\star}<t_{n-1}\right) \leq  \left(1-B(n-1, t_{n-1})\right)^n
\leq  e^{-nB(n-1, t_{n-1})}
\leq (1+o(1))e^{-\frac{(\log(n-1))^{\epsilon/2}}{\sqrt{4\pi}}}   \rightarrow 0,
\end{align}
as $n\rightarrow \infty$, appealing to the definition of $B\left(n-1, t_{n-1}\right)$,  inequality
\eqref{eq:H},  the
inequality $1-c \leq  e^{-c}$ , and relation \eqref{eq:Bn}.
\medskip

\noindent
{\bf Second Proof of $(i)$}

The second proof of $(i)$ is an application of the 2nd Moment Method frequently applied to probabilistic problems in Graph Theory \citep[Chapter 4]{AS2016}.

Let $Y_t=Y_t(T_n)$ denote the number of vertices in $T_n$  with score larger than
$t=t_{n-1}$, i.e.
$Y_t=\sum_{j=1}^{n}I\left(s_j>t\right).$
Since $P\left(s_u >t\right)=B\left(n-1, t\right)$ for any $u$ and $t$, it follows that

\begin{equation}
\label{eq:Yn}
E\left(Y_t\right)=nB(n-1, t)\sim n \frac{(\log(n-1))^{\epsilon/2}}{\sqrt{4\pi}(n-1)},
\end{equation}
upon appealing to \eqref{eq:Bn} with $t=t_{n-1}$.

To determine the variance $Var(Y_t)$ we first observe that for any ordered pair of vertices $u$ and
$v$  we can write
their scores as $s_u = s'_u+X_{uv}$ and $s'_v + X_{vu}$, where  $s'_u$  and  $s'_v$ are the number of
games $u$ and
$v$ win against the remaining $n-2$ players; since  $s'_u$  and  $s'_v$ are independent variables, it
follows that
\begin{align*}
&
P\left(s_u > t, s_v  >  t\right)\\
&
=P\left(X_{uv}=0\right)P\left(s'_u  > t\right)P\left(s'_v  >  t-1\right)+
P\left(X_{uv} = 1\right)P\left(s'_u > t-1\right)P\left(s'_v  >  t \right)\\
&
=B(n-2, t)B(n-2, t-1).
\end{align*}
Consequently,
\begin{align}
\label{eq:Vn}
&
Var(Y_t)=E(Y_t)+E(Y_t(Y_t-1))-(E(Y_t))^2 \nonumber\\
&
=E(Y_t)+n(n-1) B(n-2, t)B(n-2, t-1)-(nB(n-1, t))^2.
\end{align}
Now
\begin{align*}
&
B(n-1, t)=1/2(B(n-2, t) + B(n-2, t-1)),
\end{align*}
so relation \eqref{eq:Vn} simplifies to
\begin{align}
\label{eq:Vn1}
&
Var(Y_t)=E(Y_t)-nB(n-2, t)B(n-2, t-1)-(1/4)n^2\left\{B(n-2, t-1)-B(n-2, t)\right\}^2\nonumber\\
&
= E(Y_t)-nB(n-2, t)B(n-2, t-1)-(1/4)n^2b(n-2, t)^2    \leq  E(Y_t).
\end{align}
Therefore,
\begin{align*}
&
P\left(Y_t=0\right)\leq P\left(|Y_t-E(Y_t)|\geq E(Y_t)\right)\leq \frac{Var(Y_t)}{(E(Y_t))^2}\leq
\frac{1}{E(Y_t)}\rightarrow 0,
\end{align*}
as $n\rightarrow \infty $, by Chebyshev's Inequality, relation \eqref{eq:Vn1}, and \eqref{eq:Yn}. This
implies
conclusion (i).
\bigskip

\noindent
{\bf {Proof of $(ii)$}}

We now turn to conclusion (ii). In view of conclusion (i), we may restrict our attention to
tournaments $T_n$ in which
the maximum value $s^{\star}$ of the scores realized by the vertices is at least as large as
$t=t_{n-1}$.  Recall that
$W_n = W_n(T_n)$ denotes the number of ordered pairs of distinct vertices $u$ and $v$ of $T_n$ such
that $t<s_u=s_v$
where
$t \leq  n-1$, i.e. $$W_n=\sum_{1\leq v<u \leq n }I(t<s_u=s_{v}).$$  Let $s'_u$ and $s'_v$ denote the scores of two
such
vertices $u$ and $v$ in their matches with the remaining $n-2$ players and note that  $s'_u$ and
$s'_v$ are
independent variables. Then it follows that
\begin{align}
\label{eq:jn}
&
P\left(s_u=h, s_v=h\right)=1/2 P(s'_u=h-1)P(s'_v=h)+1/2P(s'_u=h)P(s'_v=h-1) \nonumber\\
&
={{n-2}\choose{h-1}}(1/2)^{n-2}{{n-2}\choose{h}}(1/2)^{n-2}\nonumber\\
&
=4\frac{h}{n-1}\left(1-\frac{h}{n-1}\right){{n-1}\choose{h}}(1/2)^{n-1}{{n-1}\choose{h}}(1/2)^{n-1}  \leq
\left(b(n-1,
h)\right)^2.
\end{align}
Hence,
\begin{align*}
&
E(W_n)=E\left(\sum_{1\leq v<u \leq n }I(t<s_u=s_{v})\right)=n(n-1)E\left(I(t<s_1=s_{2})\right) \\
&
=n(n-1)P\left(t<s_1=s_{2}\right)
=n(n-1)\sum^{n-1}_{h=t+1}P(s_1=h, s_2=h)
\leq n(n - 1) \sum^{n-1}_{h=t+1}b(n-1, h)^2\\
&
\leq   n(n-1) b(n-1,t+1)B(n-1, t)
\leq   n(n-1) b(n-1,t)B(n-1, t)\\
&
\sim \frac{(\log(n-1))^{1/2 + \epsilon}}{\pi\sqrt{2(n-1)}}\rightarrow 0,
\end{align*}
as $n\rightarrow \infty $. Consequently, appealing to \eqref{eq:bn}, \eqref{eq:Bn}, and to the fact that
$W_n=W_n
I(W_n>0)\geq I(W_n>0)$, we find that
\begin{equation*}
1-P(W_n=0)=P(W_n > 0)\leq E(W_n)\rightarrow 0,
\end{equation*}
 as required.
\end{proof}

\section{Remarks}
\begin{remark}
{\normalfont For the sake of completeness, we mention an upper bound that \cite{H1963} gave for the maximum score
$s^{\star}$ in almost
all tournaments $T_n$. Let $t'=t'_{n-1}$ be defined as $t=t_{n-1}$ was defined earlier except that the
$\epsilon$ in
relation \eqref{eq:xn} is replaced by $-\epsilon$ and without the ceiling function, it turns out that a relation corresponding to
\eqref{eq:Bn} is
\begin{equation*}
\label{eq:Bnn}
B(n-1, t'_{n-1})
\sim \frac{(\log(n-1))^{-\epsilon/2}}{\sqrt{4\pi}(n-1)}.
\end{equation*}
Hence,  it follows from Boole's inequality that
\begin{equation}
\label{eq:Bln}
P\left(s^{\star} >  t'\right) \leq  nB(n-1, t')=O\left((\log(n-1))^{-\epsilon/2}\right),
\end{equation}
as $n\rightarrow \infty$.
From \eqref{eq:Hln} and \eqref{eq:Bln}
\cite{H1963} concluded that
\begin{equation*}
s^{\star}-\frac{n-1}{2}-\sqrt{\frac{n-1}{4}}\sqrt{2\log(n-1)}\rightarrow 0
\end{equation*}
in probability as $n\rightarrow \infty$.
}
\end{remark}

\begin{remark}
{\normalfont
\cite{MM2022a} and \cite{MR2022} have extended Huber's inequality to a more general round-robin tournament model and to other tournaments and games models, respectively.
}
\end{remark}

\begin{remark}
{\normalfont
\cite{B1981} has derived numerous
results on the
distributions of the degree sequences $ d_1 \geq d_2 \geq \cdots \geq d_n $ of ordinary n-vertex
graphs in which edges
are present with probability $p$; see also, \cite{I1973}, \cite{B1981}, \cite{B2001}, \cite{FK2016} and \cite{M2023}.
A similar problem concerning a round-robin tournament model was considered recently in \cite{M2021}.
}
\end{remark}

\section*{Acknowledgements}
We thank Noga Alon for referring us to the work of Paul Erd\H{o}s and Robin J. Wilson.
We also thank Boris Alemi for executing the Maple program on a powerful computer and obtaining the data in the required format.
Research of Yaakov Malinovsky is supported in part by BSF grant 2020063.

{}

\end{document}